\documentclass[a4paper,12pt,english]{amsart}
\usepackage{axodraw4j}
\usepackage{amsfonts}
\usepackage{amsmath}
\usepackage{amssymb}
\usepackage{color}
\usepackage{graphicx}
\usepackage[mathcal]{eucal}
\usepackage[all]{xy}
\usepackage{xspace}
\input xy
\xyoption{all}
 \font \eightrm=cmr8

 \newcommand{\nc}{\newcommand}

\newtheorem{thm}{Theorem}
\newtheorem{exam}{Example}

\newtheorem{lem}[thm]{Lemma}
\newtheorem{prop}[thm]{Proposition}
\newtheorem{defn}{Definition}
\newtheorem{rmk}[thm]{Remark}

\def\diagramme #1{\vskip 4mm \centerline {#1} \vskip 4mm}

\nc{\BA}{{\Bbb A}} \nc{\CC}{{\Bbb C}} \nc{\DD}{{\Bbb D}}
\nc{\EE}{{\Bbb E}} \nc{\FF}{{\Bbb F}} \nc{\GG}{{\Bbb G}}
\nc{\HH}{{\Bbb H}} \nc{\LL}{{\Bbb L}} \nc{\NN}{{\Bbb N}}
\nc{\PP}{{\Bbb P}} \nc{\QQ}{{\Bbb Q}} \nc{\RR}{{\Bbb R}}
\nc{\TT}{{\Bbb T}} \nc{\VV}{{\Bbb V}} \nc{\ZZ}{{\Bbb Z}}
\nc{\Cal}[1]{{\mathcal {#1}}}
\nc{\mop}[1]{\mathop{\hbox {\rm #1} }}
\nc{\smop}[1]{\mathop{\hbox {\eightrm #1} }}
\nc{\mopl}[1]{\mathop{\hbox {\rm #1} }\limits}
\nc{\frakg}{{\frak g}}
\nc{\g}[1]{{\frak {#1}}}
\def \restr#1{\mathstrut_{\textstyle |}\raise-8pt\hbox{$\scriptstyle #1$}}
\def \srestr#1{\mathstrut_{\scriptstyle |}\hbox to
  -1.5pt{}\raise-4pt\hbox{$\scriptscriptstyle #1$}}
\nc{\wt}{\widetilde}
\nc{\wh}{\widehat}
\nc{\un}{\hbox{\bf 1}}
\nc{\redtext}[1]{\textcolor{red}{\tt #1}}
\nc{\bluetext}[1]{\textcolor{blue}{#1}}
\nc{\comment}[1]{[[{\tt {#1}}]] }
\nc{\R}{{\mathbb R}}

\nc\fleche[1]{\mathop{\hbox to #1 mm{\rightarrowfill}}\limits}

\def\semi{\mathrel{\raise 1.2pt\hbox{${\scriptscriptstyle |}$}\joinrel\kern
-3.8pt\mathrel{\times}}}

\def\racine{\,{\scalebox{0.07}{
\begin{picture}(29,29) (360,-285)
    \SetWidth{6}
    \SetColor{Black}
    \Vertex(375,-271){20}
  \end{picture}
  }}\,}
  \def\racineun{\,{\scalebox{0.5}{
  \begin{picture}(43,27) (342,-266)
    \SetWidth{1.0}
    \SetColor{Black}
    \GOval(352,-256)(9,9)(0){0.882}
    \Text(350,-260)[lb]{\Large{\Black{$1$}}}
  \end{picture}
  }}\,}
  \def\racinedeux{\,{\scalebox{0.5}{
  \begin{picture}(43,27) (342,-266)
    \SetWidth{1.0}
    \SetColor{Black}
    \GOval(352,-256)(9,9)(0){0.882}
    \Text(350,-260)[lb]{\Large{\Black{$2$}}}
  \end{picture}
  }}\,}
  \def\racinetrois{\,{\scalebox{0.5}{
  \begin{picture}(43,27) (342,-266)
    \SetWidth{1.0}
    \SetColor{Black}
    \GOval(352,-256)(9,9)(0){0.882}
    \Text(350,-260)[lb]{\Large{\Black{$3$}}}
  \end{picture}
  }}\,}
  \def\racinen{\,{\scalebox{0.5}{
  \begin{picture}(43,27) (342,-266)
    \SetWidth{1.0}
    \SetColor{Black}
    \GOval(352,-256)(9,9)(0){0.882}
    \Text(350,-260)[lb]{\Large{\Black{$n$}}}
  \end{picture}
  }}\,}
  \def\arbaf{\,{\scalebox{0.5}{
 \begin{picture}(106,125) (318,-165)
    \SetWidth{1.0}
    \SetColor{Black}
    \GOval(354,-155)(9,9)(0){0.882}
    \GOval(353,-117)(9,9)(0){0.882}
    \Line(353,-126)(353,-145)
    \GOval(351,-85)(9,9)(0){0.882}
    \Line(352,-93)(352,-109)
    \GOval(375,-62)(9,9)(0){0.882}
    \GOval(328,-61)(9,9)(0){0.882}
    \Line(333,-69)(344,-78)
    \Line(368,-68)(355,-80)
    \Text(369,-161)[lb]{\Large{\Black{$(a,1)$}}}
    \Text(368,-122)[lb]{\Large{\Black{$(e,1)$}}}
    \Text(365,-89)[lb]{\Large{\Black{$(h,2)$}}}
    \Text(389,-63)[lb]{\Large{\Black{$(d,1)$}}}
    \Text(336,-61)[lb]{\Large{\Black{$(c,2)$}}}
  \end{picture}
  }}\,}

 \def\arbae{\,{\scalebox{0.5}{
 \begin{picture}(89,130) (340,-161)
    \SetWidth{1.0}
    \SetColor{Black}
    \GOval(352,-151)(9,9)(0){0.882}
    \GOval(351,-114)(9,9)(0){0.882}
    \Line(352,-123)(352,-142)
    \GOval(350,-79)(9,9)(0){0.882}
    \GOval(350,-46)(9,9)(0){0.882}
    \GOval(381,-100)(9,9)(0){0.882}
    \Line(351,-88)(352,-103)
    \Line(373,-105)(359,-110)
    \Line(351,-54)(351,-70)
    \Text(365,-158)[lb]{\Large{\Black{$(a,1)$}}}
    \Text(362,-128)[lb]{\Large{\Black{$(e,1)$}}}
    \Text(394,-105)[lb]{\Large{\Black{$(d,1)$}}}
    \Text(361,-85)[lb]{\Large{\Black{$(h,2)$}}}
    \Text(363,-52)[lb]{\Large{\Black{$(c,2)$}}}
  \end{picture}
  }}\,}
 \def\arbad{\,{\scalebox{0.5}{
  \begin{picture}(89,130) (340,-161)
    \SetWidth{1.0}
    \SetColor{Black}
    \GOval(352,-151)(9,9)(0){0.882}
    \GOval(351,-114)(9,9)(0){0.882}
    \Line(352,-123)(352,-142)
    \GOval(350,-79)(9,9)(0){0.882}
    \GOval(350,-46)(9,9)(0){0.882}
    \GOval(381,-100)(9,9)(0){0.882}
    \Line(351,-88)(352,-103)
    \Line(373,-105)(359,-110)
    \Line(351,-54)(351,-70)
    \Text(365,-158)[lb]{\Large{\Black{$(a,1)$}}}
    \Text(362,-128)[lb]{\Large{\Black{$(e,1)$}}}
    \Text(394,-105)[lb]{\Large{\Black{$(c,2)$}}}
    \Text(361,-85)[lb]{\Large{\Black{$(h,2)$}}}
    \Text(363,-52)[lb]{\Large{\Black{$(d,1)$}}}
  \end{picture}
    }}\,}
 \def\arbac{\,{\scalebox{0.5}{
 \begin{picture}(142,104) (294,-187)
    \SetWidth{1.0}
    \SetColor{Black}
    \GOval(352,-177)(9,9)(0){0.882}
    \GOval(351,-140)(9,9)(0){0.882}
    \Line(352,-149)(352,-168)
    \GOval(349,-98)(9,9)(0){0.882}
    \GOval(304,-121)(9,9)(0){0.882}
    \GOval(389,-116)(9,9)(0){0.882}
    \Line(313,-124)(342,-136)
    \Line(349,-106)(350,-132)
    \Line(382,-122)(357,-135)
    \Text(367,-182)[lb]{\Large{\Black{$(a,1)$}}}
    \Text(364,-147)[lb]{\Large{\Black{$(e,1)$}}}
    \Text(401,-124)[lb]{\Large{\Black{$(d,1)$}}}
    \Text(360,-104)[lb]{\Large{\Black{$(h,2)$}}}
    \Text(312,-122)[lb]{\Large{\Black{$(c,2)$}}}
  \end{picture}
  }}\,}
 \def\arbab{\,{\scalebox{0.5}{
 \begin{picture}(59,57) (341,-234)
    \SetWidth{1.0}
    \SetColor{Black}
    \GOval(352,-224)(9,9)(0){0.882}
    \GOval(351,-187)(9,9)(0){0.882}
    \Line(352,-196)(352,-215)
    \Text(365,-231)[lb]{\Large{\Black{$(e,1)$}}}
    \Text(362,-200)[lb]{\Large{\Black{$(h,2)$}}}
  \end{picture}
 }}\,}
 \def\arbaa{\,{\scalebox{0.5}{
 \begin{picture}(106,88) (314,-203)
    \SetWidth{1.0}
    \SetColor{Black}
    \GOval(352,-193)(9,9)(0){0.882}
    \GOval(351,-156)(9,9)(0){0.882}
    \Line(352,-165)(352,-184)
    \GOval(324,-134)(9,9)(0){0.882}
    \GOval(375,-132)(9,9)(0){0.882}
    \Line(357,-149)(369,-138)
    \Line(333,-141)(345,-150)
    \Text(365,-199)[lb]{\Large{\Black{$(a,1)$}}}
    \Text(363,-164)[lb]{\Large{\Black{$(b,3)$}}}
    \Text(385,-141)[lb]{\Large{\Black{$(d,1)$}}}
    \Text(334,-136)[lb]{\Large{\Black{$(c,2)$}}}
  \end{picture}
    }}\,}
\def\echelunun{\,{\scalebox{0.5}{
\begin{picture}(45,63) (341,-230)
    \SetWidth{1.0}
    \SetColor{Black}
    \GOval(352,-220)(9,9)(0){0.882}
    \GOval(351,-183)(9,9)(0){0.882}
    \Line(352,-191)(352,-210)
    \Text(350,-224)[lb]{\Large{\Black{$1$}}}
    \Text(349,-187)[lb]{\Large{\Black{$1$}}}
  \end{picture}
}}\,}
\def\echelundeux{\,{\scalebox{0.5}{
\begin{picture}(45,63) (341,-230)
    \SetWidth{1.0}
    \SetColor{Black}
    \GOval(352,-220)(9,9)(0){0.882}
    \GOval(351,-183)(9,9)(0){0.882}
    \Line(352,-191)(352,-210)
    \Text(350,-224)[lb]{\Large{\Black{$1$}}}
    \Text(349,-187)[lb]{\Large{\Black{$2$}}}
  \end{picture}
}}\,}
\def\echelvw{\,{\scalebox{0.5}{
\begin{picture}(45,63) (341,-230)
    \SetWidth{1.0}
    \SetColor{Black}
    \GOval(352,-220)(9,9)(0){0.882}
    \GOval(351,-183)(9,9)(0){0.882}
    \Line(352,-191)(352,-210)
    \Text(350,-224)[lb]{\Large{\Black{$v$}}}
    \Text(349,-187)[lb]{\Large{\Black{$w$}}}
  \end{picture}
}}\,}

\def\echeldeuxun{\,{\scalebox{0.5}{
\begin{picture}(45,63) (341,-230)
    \SetWidth{1.0}
    \SetColor{Black}
    \GOval(352,-220)(9,9)(0){0.882}
    \GOval(351,-183)(9,9)(0){0.882}
    \Line(352,-191)(352,-210)
    \Text(350,-224)[lb]{\Large{\Black{$2$}}}
    \Text(349,-187)[lb]{\Large{\Black{$1$}}}
  \end{picture}
}}\,}
\def\couronne{\,{\scalebox{0.5}{
\begin{picture}(91,58) (317,-235)
    \SetWidth{1.0}
    \SetColor{Black}
    \GOval(352,-225)(9,9)(0){0.882}
    \Text(350,-229)[lb]{\Large{\Black{$1$}}}
    \GOval(327,-193)(9,9)(0){0.882}
    \GOval(374,-194)(9,9)(0){0.882}
    \Line(367,-201)(357,-217)
    \Line(333,-198)(346,-217)
    \Text(326,-199)[lb]{\Large{\Black{$1$}}}
    \Text(373,-198)[lb]{\Large{\Black{$1$}}}
  \end{picture}
}}\,}
\def\echelununun{\,{\scalebox{0.5}{
\begin{picture}(47,99) (339,-194)
    \SetWidth{1.0}
    \SetColor{Black}
    \GOval(352,-184)(9,9)(0){0.882}
    \GOval(351,-147)(9,9)(0){0.882}
    \Line(352,-155)(352,-174)
    \Text(350,-188)[lb]{\Large{\Black{$1$}}}
    \Text(349,-151)[lb]{\Large{\Black{$1$}}}
    \GOval(349,-112)(9,9)(0){0.882}
    \Line(350,-121)(351,-137)
    \Text(347,-116)[lb]{\Large{\Black{$1$}}}
  \end{picture}
}}\,}

\def\arbacc{\,{\scalebox{0.5}{
 \begin{picture}(142,104) (294,-187)
    \SetWidth{1.0}
    \SetColor{Black}
    \GOval(352,-177)(9,9)(0){0.882}
    \GOval(351,-140)(9,9)(0){0.882}
    \Line(352,-149)(352,-168)
    \GOval(349,-98)(9,9)(0){0.882}
    \GOval(304,-121)(9,9)(0){0.882}
    \GOval(389,-116)(9,9)(0){0.882}
    \Line(313,-124)(342,-136)
    \Line(349,-106)(350,-132)
    \Line(382,-122)(357,-135)
    \Text(367,-182)[lb]{\Large{\Black{$(a,1)$}}}
    \Text(364,-147)[lb]{\Large{\Black{$(e,1)$}}}
    \Text(401,-124)[lb]{\Large{\Black{$(h,2)$}}}
    \Text(360,-104)[lb]{\Large{\Black{$(d,1)$}}}
    \Text(312,-122)[lb]{\Large{\Black{$(c,2)$}}}
  \end{picture}
  }}\,}
\def\arbaccc{\,{\scalebox{0.5}{
 \begin{picture}(142,104) (294,-187)
    \SetWidth{1.0}
    \SetColor{Black}
    \GOval(352,-177)(9,9)(0){0.882}
    \GOval(351,-140)(9,9)(0){0.882}
    \Line(352,-149)(352,-168)
    \GOval(349,-98)(9,9)(0){0.882}
    \GOval(304,-121)(9,9)(0){0.882}
    \GOval(389,-116)(9,9)(0){0.882}
    \Line(313,-124)(342,-136)
    \Line(349,-106)(350,-132)
    \Line(382,-122)(357,-135)
    \Text(367,-182)[lb]{\Large{\Black{$(a,1)$}}}
    \Text(364,-147)[lb]{\Large{\Black{$(e,1)$}}}
    \Text(401,-124)[lb]{\Large{\Black{$(d,1)$}}}
    \Text(360,-104)[lb]{\Large{\Black{$(c,2)$}}}
    \Text(312,-122)[lb]{\Large{\Black{$(h,2)$}}}
  \end{picture}
  }}\,}

\def\arbaff{\,{\scalebox{0.5}{
 \begin{picture}(89,130) (340,-161)
    \SetWidth{1.0}
    \SetColor{Black}
    \GOval(352,-151)(9,9)(0){0.882}
    \GOval(351,-114)(9,9)(0){0.882}
    \Line(352,-123)(352,-142)
    \GOval(350,-79)(9,9)(0){0.882}
    \GOval(350,-46)(9,9)(0){0.882}
    \GOval(304,-100)(9,9)(0){0.882}
    \Line(351,-88)(352,-103)
    \Line(312,-99)(344,-110)
    \Line(351,-54)(351,-70)
    \Text(365,-158)[lb]{\Large{\Black{$(a,1)$}}}
    \Text(362,-128)[lb]{\Large{\Black{$(e,1)$}}}
    \Text(361,-85)[lb]{\Large{\Black{$(h,2)$}}}
    \Text(363,-52)[lb]{\Large{\Black{$(d,1)$}}}
    \Text(280,-130)[lb]{\Large{\Black{$(c,2)$}}}
\end{picture}}}\,}

  \def\angleex{\,{\scalebox{0.5}{
   \begin{picture}(398,198) (224,-239)
    \SetWidth{1.0}
    \SetColor{Black}
    \GOval(384,-218)(16,16)(0){0.882}
    \GOval(384,-122)(16,16)(0){0.882}
    \GOval(384,-122)(16,16)(0){0.882}
    \Line(384,-138)(384,-202)
    \Line(448,-90)(400,-122)
    \GOval(384,-218)(16,16)(0){0.882}
    \GOval(448,-90)(16,16)(0){0.882}
    \Line(368,-122)(320,-90)
    \GOval(320,-90)(16,16)(0){0.882}
    \Line[dash,dashsize=10](384,-106)(384,-42)
    \Line[dash,dashsize=10](464,-74)(496,-42)
    \Line[dash,dashsize=10](400,-218)(496,-202)
    \Line[dash,dashsize=10](368,-122)(256,-106)
    \Line[dash,dashsize=10](304,-90)(256,-58)
    \Line[dash,dashsize=10](400,-122)(528,-106)
    \Line[dash,dashsize=10](368,-218)(256,-202)
  \end{picture}}}\,}

\begin{document}
\title{ Weighted rooted trees and deformations of operads}

\author{ Abdellatif Sa\" idi}
\address{ Faculty of sciences of Monastir 5019, Tunisia }
\address{Laboratoire Physique math\'ematiques, fonctions sp\'etiales et applications, Sousse 4011, Tunisia}

\email{Abdellatif.Saidi@fsm.rnu.tn}

\date{May 2014}

\begin{abstract}
We will define an operad $\mathcal{B}^0$ on planar rooted trees. $\mathcal{B}^{0}$ is
analgous to the $NAP$-operad in the non-planar tree setting.
 We will define a family of "current-preserving" operads $\mathcal{B}^{\lambda}$
 depending on a scalar parameter $\lambda$,  which can be seen as a deformation of
the operad $\mathcal{B}^0$.
 Forgetting the extra "current-preserving" notion above gives back the Brace operad
for $\lambda=1$ and the $\mathcal{B}^0$ operad
 for $\lambda=0$. A natural map from non-planar rooted trees to planar ones gives back the current-preserving interpolation between $NAP$ and pre-Lie investigated in a previous article \cite{S}.
\end{abstract}
\maketitle


\tableofcontents

\section{Background on operads}


 In this section, we review the material needed for this article. We refer to
Ginzburg and
 Kapranov \cite{GLE} or J. L. Loday  \cite{L}.
 Let $K$ be a field of characteristic zero. An operad
(in the symmetric monoidal category of $k$-vector spaces) is given
by a collection of vector spaces $(\mathcal{O}(n)_{n\geq 0})$, a
right action of the symmetric group $S_n$ on $\mathcal{O}(n)$, and
a collection of compositions:

$$\begin{array}{clcll}
\circ_i&:&
\mathcal{O}(n)\otimes\mathcal{O}(m)&\longrightarrow&\mathcal{O}(n+m-1)~~~~~~~~
i=1,\dots,n;\\
&&(a,b)&\longmapsto &a\circ_i b ,\end{array}$$

with satisfy the following axioms:
\begin{itemize}
\item The two associativity conditions:
$$\begin{array}{cccc}
a\circ_i (b\circ_j c)&=&(a\circ_i b)\circ _{i+j-1}c,& \forall
i\in\{1,\dots,n\},~~\forall j\in\{1,\dots, m\}\\
(a\circ_i b)\circ_{m+j-1}c&=&(a\circ_j c)\circ_i b,&~~~\forall
i,j\in\{1,\dots,n\}, i<j,
 \end{array}$$
called nested associativity and disjoint associativity
respectively.
\item The unit axiom: there exists an object $e\in\mathcal{O}(1)$
for which for any $a\in \mathcal{O}(n)$:
$$\begin{array}{ccll}
e\circ a&=&a,\\
a\circ_i e&=& a,~~~\forall i\in\{1,\dots,n\}
\end{array}$$
\item The equivariance axiom:

for any $\sigma\in S_n,~~\tau \in S_m$, we have:
$$a.\sigma \circ_{\sigma(i)}b.\tau=(a\circ_i b).\rho(\sigma,\tau),$$
where $\rho(\sigma,\tau)\in S_{n+m-1}$ is defined by letting $\tau$ permute the set
$E_i = \{i, i +
1,\cdots, i+m-1\}$ of cardinality $m$, and then by letting $\sigma$ permute the set
$\{1,\cdots, i-1,Ei,i +m,\cdots, m + n-1\}$
 of cardinality $n$.
\end{itemize}
\begin{exam}
Any vector space $V$ yields an operad $End_V$ with $End_V (n)=
Hom(V^{\otimes n},V)$ where:
$$(f\circ_i
g)\big(a_1,\dots,a_{n+m-1}\big)=f\big(a_1,\dots,a_{i-1},g(a_i,\dots,a_{i+m-1}),a_{i+m},\dots,a_{n+m-1}\big)$$
\end{exam}
An algebra over an operad $\mathcal{O}$, or $\mathcal{O}$-algebra,
is a vector space $V$ together with an operad morphism from
$\mathcal{O}$ to $End_V$. This is equivalent to giving linear maps
$$\mathcal{O}(n)\otimes_{S_n} V^{\otimes n}\rightarrow V,$$
satisfying associativity conditions with  respect to the
compositions.\\
An important point in the theory of operad
is the following theorem :
\begin{thm}
\cite[Chapter 5, sect 5.7.1]{LV}. The free $\mathcal{O}$-algebra generated by $V$ is the space
$\mathcal{O}(V)=\bigoplus_{n\geq 0}\mathcal{O}(n)\otimes_{S_n}
V^{\otimes n}.$
\end{thm}

In the remainder of this article, we describe operads
by species formalism, i.e: we replace the set $\{1,2,...,n\}$
by any finite set $A$ of cardinal $n$. For more details see sections $2$ and $3$ of
\cite{S}.

\section{A description of Pre Lie and Brace operads}

\subsection{Rooted trees and planar rooted trees}

\begin{itemize}
\item A rooted tree $T$ is a finite graph, without loops, with a
special vertex called the root of $T$. The set of rooted trees
will be denoted by $\mathcal{T}$. Let $D$ be a nonempty set. A
rooted tree decorated by $D$ is a rooted tree with an application
from the set of its vertices into $D$. The set of rooted trees
decorated by $D$ will be denoted by $\mathcal{T}^{D}$. Following
the notation of Connes and Kreimer \cite{CK1}, any tree $T$ writes
$T=B_{+}(r,T_1\cdots T_k)$ where $r$ is the root (or the decoration
of the root) and $T_1,\dots, T_k$ are trees. So we have
$$B_{+}(r,T_1\cdots T_k)=B_{+}(r,T_{\sigma(1)}\cdots
T_{\sigma(k)})~~~\forall \sigma\in S_k.$$
The vector space spanned
by $\mathcal{T}^D$ will be denoted by $\mathcal{RT}^D$. We denote by $\mathcal{RT}$ the species of rooted trees: for any finite set $A$ the vector space $\mathcal{RT}(A)$ is spanned by the rooted trees with $\vert A\vert$ vertices, together with a bijection from the set of vertices onto $A$.
\item  A planar  rooted tree $T$ is a rooted tree with an
embedding into the plane. The set of planar rooted trees will be
denoted by $\mathcal{T}_P$. Let $D$ be a nonempty set. A planar
rooted tree decorated by $D$ is a planar tree with an application
from the set of its vertices into $D$. The set of planar rooted
trees decorated by $D$ will be denoted by $\mathcal{T}_{P}^{D}.$
If $T_1 \neq T_2$, then:
$$B_+ (r,T_1T_2\cdots T_k)\neq B_+ (r,T_2 T_1\cdots T_k).$$
We draw the planar tree in the disk:
$$\mathcal{D}_+ =\{(x,y)\in\mathbb{R}^2; y>0~~\text{and}~~ x^2 +y^2<1\}~,$$
but the root is drawn in $x=y=0$.

The vector space spanned by $\mathcal{T}_{P}^{D}$ will be denoted
by $\mathcal{PRT}^D$.We denote by $\mathcal{PRT}$ the species of planar rooted trees: for any finite set $A$ the vector space $\mathcal{PRT}(A)$ is spanned by the rooted trees with $\vert A\vert$ vertices, together with a bijection from the set of vertices onto $A$.  
\end{itemize}

\subsection{The pre-Lie operad}

We describe the pre-Lie operad in terms of non-planar labelled rooted trees, following
\cite{ChaLiv}. Let $A$ and $B$ be two finite sets. Let $v\in A$.
We define the partial composition
$\circ_v:\mathcal{RT}(A)\otimes\mathcal{RT}(B)\rightarrow\mathcal{RT}\big((A-\{v\})\amalg
B\big)$, as follows:
\begin{equation}
T\circ_v S=\sum_{f: E(T,v)\rightarrow B}{T\circ_{v}^{f}S},
\end{equation}
where $T\circ_{v}^{f}S$ is the tree of
$\mathcal{RT}((A-\{v\})\amalg B\big)$ obtained by replacing the
vertex $v$ of $T$ by the tree $S$ and connecting each edge $a$ in
$E(T,v)$ at the vertex $f(a)$ of $S$. If $v$ is not the root of
$T$, the edge going down from $v$ is now going down from the root
of $S$. The root of the new tree is the root of  $T$ if it is
different from vertex $v$, and of  $S$ else (see details in
 \cite{ChaLiv}). The unit is the tree with a single vertex.
  These partial compositions define an operad which is the pre-Lie operad.

\subsection{The Brace operad}

We describe the brace operad by the planar labelled rooted trees
 (for more details see \cite{Chap1}). Let $T$ be a labelled planar rooted
tree. Let $s$ be a vertex of $T$. Let $B_{e}(s)$ be a little
disk of center $s$. The pair $(s,\alpha)$ is called an angle of
$T$ if $\alpha$ is a connected  component  of $B_{e}(s)\bigcap
(\mathcal{D}_+ \backslash T).$ We denote by $Ang(T)$ the set of
angles of $T$. Naturally, from left to right we set a total
order on $Ang(T)$ as follows: considering an angle as a direction from a vertex,
one can draw a path from every angle to a point of the upper part of the unit
circle. We order then these points clockwise.\\

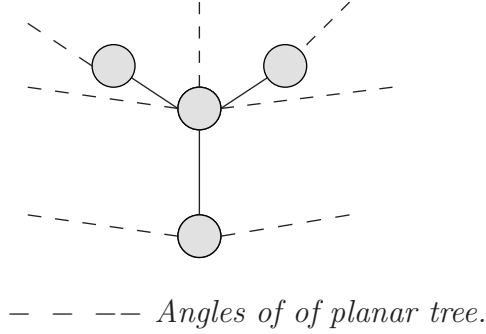
\begin{exam}
$$\angleex$$
\begin{center}
$-~-~--$ Angles of of planar tree.
\end{center}
\end{exam}
Let $T, S$  be  labelled planar rooted trees. Let $v$ be a
vertex of $T$. We denote by $E(T,v)$ the totally ordered set
(from left to right)  of the incoming edges on $v$. We can
consider the set of increasing functions from $E(T,v)$ to $Ang
(S)$. We define
\begin{equation}
T\diamond_v S=\sum_{f:E(T,v)\to Ang(S)}{T\diamond_{v}^{f}S},
\end{equation}
where $f$ is an increasing function and $T\diamond_{v}^{f}S$ is the
planar tree obtained by substitution  of $S$ on vertex  $v$ of
$T$, plugging the incoming edges on $S$ according to map $f$.\\
These partial compositions defined above, define a structure of an
operad which is the Brace operad \cite{Chap1}.

\section {A description of the NAP-operad  and its counterpart in the planar rooted
trees setting }

\subsection{The NAP operad}

we describe the NAP operad by the non-planar labelled rooted
trees \cite{Liv}. We define  the partial compositions $\circ_v
:\mathcal{RT}(A)\otimes\mathcal{RT}(B)\rightarrow\mathcal{RT}\big((A-\{v\})\amalg
B\big)$, as follows:
\begin{equation}
T\circ_v S={T\circ_{v}^{f_0}S},
\end{equation}
where $T\circ_{v}^{f_0}S$ is the labelled rooted tree of
$\mathcal{RT}\big((A-\{v\})\amalg B\big)$ obtained by replacing
the vertex $v$ of $T$ by the tree $S$ and connecting each edge $a$
in $E(T,v)$ at the root of $S$. The unit is the tree with a single
vertex.

\subsection{An operad of planar rooted trees analogous to NAP }

In this section, we describe an operad of  planar labelled
rooted trees. This operad is the planar analogue  of NAP.\\
 We define the partial compositions
$\diamond_v
:\mathcal{PRT}(A)\otimes\mathcal{PRT}(B)\rightarrow\mathcal{PRT}\big((A-\{v\})\amalg
B\big)$, as follows:
\begin{equation}
T\diamond_v S=\sum_{f_0:E(T,v)\to Ang^0(S)}{T\diamond_{v}^{f}S},
\end{equation}
where $Ang^0(S)$ is the set of angles starting from the root of $S$ and $f_0$ is
an increasing function from $E(T,v)$ to $Ang^0(S)$.
\begin{prop}
The partial compositions introduced above define a structure of an
operad. We will denote this operad by $\mathcal{B}^0$. The unit is the tree with a single vertex.
\end{prop}
\begin{proof}
We easily verify the unity, associativity and equivariance axioms. We omit the proof, as we will give a proof of a more general result later on.
\end{proof}

\begin{defn} \label{symm}
We will denote the symmetrization operator of trees by $\varphi$
from  the space of non-planar labelled rooted trees to the space of planar labelled rooted trees,
by induction we define $\varphi$:\\
$\varphi(\racine)=\racine$ and if $T=B_+ (r,T_1\cdots T_k)$,
then
$$\varphi(T)=B_+\big(r,\varphi(T_1 )\sqcup \varphi(T_2)\sqcup\cdots\sqcup
\varphi(T_k)\big),$$ where $\sqcup$ is the shuffle product.
\end{defn}
i.e: $\varphi(T)$ is the sum of all planar representations of $T$.
\begin{thm}
$\varphi$ is a morphism of operads from Pre-Lie to Brace \cite{Chap1}.
Similarly $\varphi$ is a morphism of operads from NAP to $\mathcal{B}^0$.
\end{thm}
\begin{proof}
The first assertion is proved by F. Chapoton \cite[Prop 4]{Chap1}, the second assertion  follows immediately by considering only the terms of minimal potential energy (see Definition \ref{potentiel} below).
\end{proof}
\begin{defn} We define
\begin{equation}
T\star S=(\echelvw\diamond_v T)\diamond_w S.
\end{equation}
Equivalently, if $T=B_+ (r,T_1\ldots T_n)$ then
$$T\star S=\sum_{i=0}^{n}{B_+ (r,T_1\ldots T_i ST_{i+1}\ldots T_n)}.$$
\end{defn}
\begin{prop}
The space $(\mathcal{T}_{P}^{D},\star)$ is a right non-associative permutative
algebra \cite{F}. i.e: for any planar rooted trees $T,S,U$, we have:
$$(T\star S)\star U=(T\star U)\star S$$
\end{prop}

 \section{The notion of current-preserving operads}

\subsection{Structure of current-preserving operads}

Let $\mathcal{O}$ be an operad and $G$ be a commutative semigroup,
 with additively denoted binary law. We say that $\mathcal{O}$  has a structure
 of $G$-current-preserving operad, if moreover
$\mathcal{O}_A=\prod\mathcal{O}_{A,W},~~~W:A\to G$ where :
 \begin{itemize}
 \item The right action of the symmetric group $AutA$ verifies:
 $$\mathcal{O}_{A,W}.\sigma=\mathcal{O}_{A,W\circ\sigma},~~\forall  \sigma\in AutA.$$
 \item For any finite sets $A, B$ and $v\in A,$ we have:
 $$\circ_v: \mathcal{O}_{A,W}\otimes\mathcal{O}_{B,X}\to\mathcal{O}_{A-\{v\}\amalg
B,W\amalg X},$$
 with image zero if $\sum_{b\in B}{X(b)}\neq W(v)$. Here, $W\amalg X$
 is defined by: $W\amalg X(a)=W(a),\forall a\in A$, and $W\amalg X(b)=X(b),~~\forall
b\in B$.
 \end{itemize}
 \begin{exam}
 The model of current-preserving operads is the operad $Endop(V)$ where $V$ is a $G$-graded
vector space \cite{S}.
 Current-preserving operads are colored operads, with an extra structure given by
semigroup law on the set of colors.
 \end{exam}

 \subsection{Current-preserving operads associated to ordinary operads} \label{ordinaryoperad}

 Given an operad $\mathcal{O}$ and any commutative semigroup $G$, we define
 a $G$-current-preserving operad $\mathcal{O}^{G}$ as follows: for any finite 
 set $A$, we have:
 $$\mathcal{O}_{A}^{G}:=\prod_{W:A\to G}{\mathcal{O}_{A,W},}$$
 where $\mathcal{O}_{A,W}$ is nothing but a copy of $\mathcal{O}_A$. The partial compositions 
 of $(\alpha,W)\in \mathcal{O}_{A,W}$ and $(\beta,X)\in \mathcal{O}_{B,X}$ are defined for any $a\in A$ by:
 $$(\alpha,W)\circ_{a}^{G}(\beta,X)=(\alpha\circ_{a}\beta,W\amalg X),$$
 if $\sum_{b\in B}{X(b)}=W(a)$, and $(\alpha,W)\circ_{a}^{G}(\beta,X)=0$ 
 if $\sum_{b\in B}{X(b)}\neq W(a)$. There is a natural morphism of operads 
 (in the ordinary sense) $\phi^G :\mathcal{O}\rightarrow \mathcal{O}^{G}$ given for
 any finite set $A$ and for any $\alpha\in \mathcal{O}_A$ by:
 $$\phi^G (\alpha):=\sum_{W:A\to G}{(\alpha,W)}.$$
 The algebras on $\mathcal{O}^G$ are nothing but $G$-graded algebras on $\mathcal{O}$.
  The morphism of operads $\phi$ simply reflects the forgetful functor from $G$-graded $\mathcal{O}$-algebras 
  to $\mathcal{O}$-algebras. Hence, any ordinary operad gives rise to a $G$-current-preserving operad naturally associated with it.

\section{A family of current-preserving operads}

\subsection{Interpolation between NAP and Pre-Lie}\label{NAPP}

We give here a family of $\mathbb N^*$-current-preserving operads
$(\mathcal{O}^{\lambda})_{\lambda\in K}$, where $\mathbb N^*$ is
the additive semi-group $\{1,2,3,\ldots\}$ of positive integers.
We have shown that this family interpolates between the $\mathbb
N^*$-current-preserving version of the NAP operad and the $\mathbb
N^*$-current-preserving version of the pre-Lie operad (see details in \cite{S}).

 \begin{defn}
 We introduce non-planar rooted trees with weights on their vertices:
 $$\racineun ,\racinedeux ,\racinetrois,\ldots$$
For a non-planar rooted tree $T$ and a weight function $W:v\mapsto
W(v)\in\mathbb N^*$, we define the weight of  $(T,W)$ by:
 \begin{equation}
 \left|T\right|=\sum_{v\in v(T)}W(v),
 \end{equation}
 where $v(T)$ denotes the set of vertices of $T$. Sometimes we will also use the
notation $|v|$ instead of $W(v)$.
\begin{exam}
 $\racineun ,\racinedeux ,\racinetrois
,\echelunun,\echelundeux,\echeldeuxun,\couronne,\echelununun$ are
 the non-planar rooted trees with weight less or equal to $3$.
 \end{exam}
 \end{defn}
 We draw non-planar rooted trees with labels and numbers on their vertices, each number refers
to the weight of the vertex.
 \begin{defn} \label{potentiel}
 We define the potential energy of a weighted non-planar rooted tree $(T,W)$ by:
 \begin{equation}
 d(T)=\sum_{v\in v(T)}W(v)h(v),
 \end{equation}
 where $h(v)$ is the height of $v$ in $T$, i.e. the distance from $v$ to the root of
$T$ counting the number of edges.
 \end{defn}
 This notion of potential energy matches the physical intution: if a branch
  of a tree is moved down, the potential energy decreases by a multiple of its
weight.\\

For any finite set $A$, let $\mathcal{O}_A$ be the completed
vector space spanned by the non-planar rooted trees with $\left|A\right|$
vertices of any weight, labellized by $A$. Namely:
\begin{equation}
\mathcal{O}_A:=\prod_{W:A\to\mathbb N^*}\mathcal{O}_{A,W},
\end{equation}
where $\mathcal{O}_{A,W}$ is the vector space spanned by the
non-planar rooted trees with $\left|A\right|$ vertices, labellized by $A$ and
with weight function $W$. For any weighted non-planar rooted tree $S\in
\mathcal{O}_{A}$ and any vertex $v$ of $S$, $E(S,v)$ denotes the
set of edges of $S$ arriving at the vertex  $v$ of $S$. Let $B$ be
another finite set and $T\in \mathcal{O}_{B}$ another weighted
rooted tree with $|B|$ vertices. Let $\lambda$ be an element of
the field $K$. We define the partial compositions by:
$$S\circ_{v,\lambda} T=\left\{\begin{array}{l}
\sum_{f:E(S,v)\rightarrow v(T)}{\lambda^{d(S\circ_{v}^{f}
T)-d(S\circ_{v}^{f_0}T)}S\circ_{v}^{f}T}~~~~\text{if}~~\left|~T\right|=\left|v\right|\\
 0~~~~~\text{otherwise},
\end{array}\right.$$
where  $S\circ_{v}^{f}T$ is the element of
$\mathcal{O}_{(A-\{v\})\amalg B}$ obtained by replacing the vertex
$v$ by $S$ and connecting each edge of $E(S,v)$ to its
 image by $f$ in $v(T)$. Here  $f_0$ is the map from $E(S,v)$ to
$v(T)$ which sends each edge $a$ of $E(S,v)$ to the root of $T$.
The tree $S\circ_{v}^{f_0}T$  has therefore the smallest potential
energy in the above sum. Unit is given by:
$$e=\sum_{n\geq1}{\racinen},$$
where $\racinen$ is the tree with one single vertex of weight $n$
(this infinite sum makes sense as ${\mathcal O}_1$ is a direct
product). The right action of the symmetric groups is given by
permutation of the labels.

\begin{exam}
Let us consider $S=\arbaa$ and $T=\arbab$. Here letters $a,b,c...$ are labels of
vertices, which are of weight $1,2$ or $3$.  We have:
 \begin{align*}
S\circ_{b,\lambda} T=\\
&&\arbac&+\lambda\arbad&+\lambda^{2}\arbae  &\hskip 6mm+\lambda^{3}\arbaf\\
          & &\hskip -8mm f_0
          :\left\{\begin{array}{ccc}c&\rightarrow&e\\d&\rightarrow&e\end{array}\right.\hbox
to
10mm{}&f:\left\{\begin{array}{ccc}c&\rightarrow&e\\d&\rightarrow&h\end{array}\right.&f:\left\{\begin{array}{ccc}c&\rightarrow&h\\d&\rightarrow&e\end{array}\right.&\hskip
8mm
f:\left\{\begin{array}{ccc}c&\rightarrow&h\\d&\rightarrow&h\end{array}\right.&&
\end{align*}
\end{exam}
\begin{thm}\label{principal}
The partial compositions defined above \cite{S} yield a structure of $\,\mathbb
N^*$-current-preserving operad on the species $A\mapsto \Cal O_A$, denoted by
$\mathcal{O}^{\lambda}$.
\end{thm}
\begin{rmk}
With the notations of Sec \ref{ordinaryoperad}, the current-preserving operad $\mathcal{O}^{\lambda}$ is naturally associated with $NAP$ for $\lambda=0$
and with Pre-Lie for $\lambda=1$ \cite{S}.
\end{rmk}

\subsection{Interpolation between Brace and $\mathcal{B}^0$}

We keep the notations of the sec. \ref{NAPP} but we replace non-planar rooted trees by planar
 rooted trees and $\mathcal{O}$ by  $\mathcal{B}$.\\
 Let $A$ be a finite set. Let $S$ be a weighted planar rooted tree with $|A|$ vertices.
  Let $B$ another finite set and $T$ another weighted planar rooted tree with $|B|$
vertices.
  Let $\lambda$ be an element of the field $K$. We define the partial compositions:
  $$S\diamond_{v,\lambda} T=\left\{\begin{array}{l}
\sum_{f:E(S,v)\rightarrow~ \text{Ang}(T)}{\lambda^{d(S\diamond_{v}^{f}
T)-d(S\diamond_{v}^{f_0}T)}S\diamond_{v}^{f}T}~~~~\text{if}~~\left|~T\right|=\left|v\right|\\
 0~~~~~\text{otherwise},
\end{array}\right.$$
where $S\diamond_{v}^{f}T$ is the weighted planar rooted tree of
$\mathcal{B}_{A-\{v\}\amalg B}$
obtained by replacing the vertex $v$ by $T$ and connecting each edge of $E(S,v)$ to
its image
by $f$ in $Ang(T)$. Here, $f_0$ is any increasing map from $E(S,v)$ to $Ang^0(T)$.
 The trees $S\diamond_{v}^{f_0}T$
 have, therefore,
 the smallest potential energy in the above sum.
 \begin{rmk}
 $f_0$ is not unique but the energy $d(S\diamond_{v}^{f_0}T)$ is the same for any $f_0
:E(S,v)\to Ang^0(T).$
 \end{rmk}

 \begin{exam}

 Let us consider $S=\arbaa$ and $T=\arbab$. Here letters $a,b,c...$ are labels of
vertices,
  which are of weight $1,2$ or $3$.  We have:
 \begin{align*}
S\diamond_{b,\lambda} T=\\
&&\arbacc&+ \arbac& +\arbaccc\\
&& +\lambda\arbaff&+ \lambda^{2}\arbae& +\lambda^{3}\arbaf
\end{align*}
 \end{exam}

\begin{thm}
The partial compositions defined above yield a structure
of $\mathbb{N}^*$- current-preserving operad on the species $A\longmapsto
\mathcal{B}_A$,
 denoted by $\mathcal{B}^{\lambda}$.
\end{thm}

\begin{proof}

We prove nested associativity first, and then disjoint associativity.
\begin{itemize}
\item \textbf{Nested associativity:}

Let $S,T,U$ be three weighted planar trees, let $v$ be a vertex of $S$
and $w$ be a vertex of
 $T$ such that $\left|T\right|=\left|v\right|$ and
$\left|U\right|=\left|w\right|$.\\

 Show $(S\diamond_{v,\lambda} T)\diamond_{w,\lambda} U=S\diamond_{v,\lambda}
(T\diamond_{w,\lambda} U)$
 where $v$ is a vertex of $S$ and $w$ a vertex of $T$.\\
 We have:
\begin{eqnarray*}
(S\diamond_{v,\lambda} T)\diamond_{w,\lambda} U&=&\sum_{f:E(S,v)\rightarrow
Ang(T)}{\lambda^{d(S\diamond_{v}^{f}T)-d(S\diamond_{v}^{f_0}T)}(S\diamond_{v}^{f}T)\diamond_{w,\lambda}
U}\\
&\hskip -20mm=&\hskip -10mm \sum_{f:E(S,v)\rightarrow
Ang(T)}{\sum_{g:E(S\diamond_{v}^{f}T, w)\rightarrow
Ang(U)}{\lambda^{d(S\diamond_{v}^{f}T)-d(S\diamond_{v}^{f_0}T)+d((S\diamond_{v}^{f}T)\diamond_{w}^{g}U)-d((S\diamond_{v}^{f}T)\diamond_{w}^{g_0
}U)}(S\diamond_{v}^{f}T)\diamond_{w}^{g}U}}\\
&=&\sum_{f:E(S,v)\rightarrow Ang(T)}{\sum_{g:E(S\diamond_{v}^{f}T,w)\rightarrow
Ang(U)}{\lambda^{A(f,g)}(S\diamond_{v}^{f}T)\diamond_{w}^{g}U}},
\end{eqnarray*}
where:
\begin{equation}
A(f,g)=d(S\diamond_{v}^{f}T)-d(S\diamond_{v}^{f_0}T)+d\big((S\diamond_{v}^{f}T)\diamond_{w}^{g}U\big)-d\big((S\diamond_{v}^{f}T)\diamond_{w}^{g_0}U\big).
\end{equation}
 Similarly we have:
\begin{eqnarray*}
S\diamond_{v,\lambda}(T\diamond_{w,\lambda}U)&=&\sum_{\wt g:E(T,w):\rightarrow
Ang(U)}{\lambda^{d(T\diamond_{w}^{\wt g}U)-d(T\diamond_{w}^{\wt
g_0}U)}S\diamond_{v,\lambda}(T\diamond_{w}^{\wt g}U)}\\
&\hskip -30mm=&\hskip -15mm \sum_{\wt f :E(S,v)\rightarrow Ang(T\diamond_{w}^{\wt
g}U)}\ {\sum_{\wt g :E(T,w)\rightarrow Ang(U)}{\lambda^{d(T\diamond_{w}^{\wt
g}U)-d(T\diamond_{w}^{\wt g_0}U)+d\big(S\diamond_{v}^{\wt f}(T\diamond_{w}^{\wt
g}U)\big)-d\big(S\diamond_{v}^{\wt f _{0}}(T\diamond_{w}^{\wt
g}U)\big)}S\diamond_{v}^{\wt f}(T\diamond_{w}^{\wt g}U)}}\\
&=&\sum_{\wt g : E(T,w)\rightarrow Ang(U)}\ {\sum_{\wt f :E(S,v)\rightarrow
Ang(T\diamond_{w}^{\wt g}U)}{\lambda^{B(\wt f,\wt g)}}S\diamond_{v}^{\wt
f}(T\diamond_{w}^{\wt g}U)},
\end{eqnarray*}
where we have set:
\begin{equation}
B(\wt f,\wt g)=d(T\diamond_{w}^{\wt g}U)-d(T\diamond_{w}^{\wt
g_0}U)+d\big(S\diamond_{v}^{\wt f}(T\diamond_{w}^{\wt
g}U)\big)-d\big(S\diamond_{v}^{\wt f_{0}}(T\diamond_{w}^{\wt g}U)\big).
\end{equation}
In order to show $(S\diamond_{v,\lambda} T)\diamond_{w,\lambda}
U=S\diamond_{v,\lambda} (T\diamond_{w,\lambda} U)$, we have to prove the following
lemma:
\begin{lem}\label{lemme emboite}
There is a natural bijection $(f,g)\longmapsto (\wt f,\wt g)$ such that
\begin{equation}\label{equation emboite}
(S\diamond_{v}^{f}T)\diamond_{w}^{g}U=S\diamond_{v}^{\wt f}(T\diamond_{w}^{\wt g}U).
\end{equation}

\end{lem}
\begin{proof}
Let $v$ be a vertex of $S$ and $w$ be a vertex of $T$ such that
$\left|T\right|=\left|v\right|$ and $\left|U\right|=\left|w\right|$.
We denoted by $Ang_w(T)$ the set of angles of $T$ issued from the vertex $w$.
Let $f:E(S,v)\rightarrow Ang(T)$ and $g:E(S\circ_{v}^{f}T,w)\rightarrow Ang(U)$ be a
two increasing functions.
We look for $\wt g:E(T,w)\rightarrow Ang(U)$ and $\wt f:E(S,v)\rightarrow
Ang(T\diamond_{w}^{\wt g}U)=Ang(U)\cup Ang(T)\backslash\{Ang_w(T)\}$
such that the equation \eqref{equation emboite} is checked.\\
Let $e$ be an edge of $T$ arriving at $w$, thus $e$ is an edge of $S\diamond_v T$
arriving at $w$. We set $\tilde{g}(e)=g(e)$. Similarly we define $\wt f$ in a unique
way:

$$\begin{array}{ccccl}
\wt f&:&E(S,v)&\longrightarrow &Ang(T\diamond_{w}^{\wt g}U)=Ang(U)\cup
Ang(T)\backslash\{Ang_w(T)\}\\
         & & e    &\longmapsto     &\wt f(e)=\left\{\begin{array}{ccc}f(e)& if
&f(e)\notin Ang_w(T)\\g(e)&if& f(e)\in Ang_w(T)\end{array}\right.
\end{array}$$
Conversely, we assume that we have the pair $(\wt f,\wt g)$ and look for the pair
$(f,g)$ such that
 equation \eqref{equation emboite} is verified.
 We have $\wt f:E(S,v)\rightarrow Ang(U)\cup Ang(T)\backslash\{Ang_w(T)\}$ and $\wt
g:E(T,w)\rightarrow Ang(U).$
 We then define:
$$\begin{array}{ccccl}
f&:&E(S,v)&\longrightarrow &Ang(T)\\
 & & e    & \longmapsto    &f(e)=\left\{\begin{array}{ccc}\wt f(e)&if&\wt f(e)\notin
Ang(U)\\ w&if& \wt f(e)\in Ang(U)\end{array}\right.
\end{array}$$
and
$$\begin{array}{ccccl}
g&:&E(S\diamond_{v}^{f}T,w)&\longrightarrow&Ang(U)\\
 & &e                  &\longmapsto
 &g(e)=\left\{\begin{array}{ccl}\wt g(e)&if&e\in T\\\wt f(e)&if&e\in S\hbox{ and
}f(e)\in Ang_w(T).\end{array}\right.
\end{array}$$
\end{proof}

\textit{Proof of Theorem \ref{principal} (continued) :} To show
$(S\diamond_{v,\lambda} T)\diamond_{w,\lambda} U=S\diamond_{v,\lambda}
(T\diamond_{w,\lambda} U)$,
 it remains to show the equality $A(f,g)=B(\wt f,\wt g)$. We set:
\begin{eqnarray*}
A'(f,g)&=&A(f,g)+d((S\diamond_{v}^{f_0}T)\diamond_{w}^{g_0}U),\\
B'(\wt f,\wt g)&=&B(\wt f,\wt g)+d(S\diamond_{v}^{\wt f_{0}}(T\diamond_{w}^{\wt
g_{0}}U)),\\
\epsilon(f)&=&d(S\diamond_{v}^{f}T)-d(S\diamond_{v}^{f_0}T),\\
\epsilon(\wt g)&=&d(T\diamond_{w}^{\wt g}U)-d(T\diamond_{w}^{\wt g_0}U).
\end{eqnarray*}
We have:
\begin{equation}
\epsilon(f)=\sum_{e\in E(S,v)}{h\big(f(e)\big).\left|B_e\right|},
\end{equation}
where $h\big(f(e)\big)$ is the distance between $f(e)$ and the root of $T$ in the
new tree $S\diamond_{v}^{f}T$
 and $\left|B_e\right|$ is the weight of the branch above $e$. Similarly:
\begin{eqnarray*}
d\big((S\diamond_{v}^{f}T)\diamond_{w}^{g_0}U\big)-d\big((S\diamond_{v}^{f_0}T)\diamond_{w}^{g_0}U\big)&=&\sum_{e\in
E(S,v)}{h\big(f(e)\big)\left|B_e\right|}\\
&=&\epsilon(f).
\end{eqnarray*}
Here $g_0$ is not involved because it was connected to the root of $U$, so
$A'(f,g)=d\big((S\diamond_{v}^{f}T)\diamond_{w}^{g}U\big)$.
 By the same computation with $\wt g$ instead of $f$ we
show $B'(\wt f, \wt g)=d\big(S\diamond_{v}^{\wt f}(T\diamond_{w}^{\wt g}U)\big)$. So by
Lemma \ref{lemme emboite} we have  $A'(f,g)=B'(\wt f, \wt g)$ that is to say
$A(f,g)+d\big((S\diamond_{v}^{f_0}T)\diamond_{w}^{g_0}U\big)=B(\wt f, \wt
g)+d\big(S\diamond_{v}^{\wt f_0}(T\diamond_{w}^{\wt g_0}U)\big)$,
which proves that $A(f,g)=B(\wt f, \wt g)$ because by Lemma \ref{lemme emboite} we have
$d\big((S\diamond_{v}^{f_0}T)\diamond_{w}^{g_0}U\big)=d\big(S\diamond_{v}^{\wt
f_0}(T\diamond_{w}^{\wt g_0}U)\big).$\\

\item  \textbf{Disjoint associativity:} let $v,w$ be two disjoint vertices of $S$
such that $\left|v\right|=\left|T\right|$ and $\left|w\right|=\left|U\right|$, show
that:
\begin{equation}\label{equation disjointe}
(S\diamond_{v,\lambda} T)\diamond_{w,\lambda} U=(S\diamond_{w,\lambda}
U)\diamond_{v,\lambda} T.
\end{equation}
We have
\begin{eqnarray*}
(S\diamond_{v,\lambda} T)\diamond_{w,\lambda} U&=&\sum_{f:E(S,v)\rightarrow
Ang(T)}{\lambda^{d(S\diamond_{v}^{f}T)-d(S\diamond_{v}^{f_0}T)}(S\diamond_{v}^{f}T)\diamond_{w,\lambda}
U}\\
&=&\sum_{f:E(S,v)\rightarrow Ang(T)}{\sum_{g:E(S\diamond_{v}^{f}T,w)\rightarrow
Ang(U)}{\lambda^{k(f)+d\big((S\diamond_{v}^{f}T)\diamond_{w}^{g}U\big)-d\big((S\diamond_{v}^{f}T)\diamond_{w}^{g_0}U\big)}}(S\diamond_{v}^{f}T)\diamond_{w}^{g}U}\\
&=& \sum_{f:E(S,v)\rightarrow Ang(T)}{\sum_{g:E(S\diamond_{v}^{f}T,w)\rightarrow
Ang(U)}{\lambda^{C(f,g)}}(S\diamond_{v}^{f}T)\diamond_{w}^{g}U},
\end{eqnarray*}
where:
\begin{eqnarray*}
k(f)&=&d(S\diamond_{v}^{f}T)-d(S\diamond_{v}^{f_0}T),\\
C(f,g)&=&k(f)+d\big((s\diamond_{v}^{f}T)\diamond_{w}^{g}U\big)-d\big((S\diamond_{v}^{f}T)\diamond_{w}^{g_0}U\big).
\end{eqnarray*}
Similarly we find:
$$(S\diamond_{w,\lambda} U)\diamond_{v,\lambda}T=\sum_{\wt g:E(S,w)\rightarrow
Ang(U)}{\sum_{\wt f:E(S\diamond_{w}^{\wt g}U,v)\rightarrow Ang(T)}{\lambda^{D(\wt f,
\wt g)}}(S\diamond_{w}^{\wt g}U)\diamond_{v}^{\wt f}T},$$
where:
\begin{eqnarray*}
D(\wt f,\wt g)&=&k(\wt g)+d\big((S\diamond_{w}^{\wt g}U)\diamond_{v}^{\wt
f}T\big)-d\big((S\diamond_{w}^{\wt g}U)\diamond_{v}^{\wt f_{0}}T\big), \hbox{
with}\\
k(\wt g)&=&d(S\diamond_{w}^{\wt g}U)-d(S\diamond_{w}^{\wt g_0}U).
\end{eqnarray*}
In order to prove \eqref{equation disjointe}, we need the following lemma:

\begin{lem}\label{lemme disjointe}
We have a natural bijection $(f,g)\longmapsto (\wt f,\wt g)$ such that
\begin{equation}\label{equa}
(S\diamond_{v}^{f}T)\diamond_{w}^{g}U=(S\diamond_{w}^{\wt g}U)\diamond_{v}^{\wt f}T.
\end{equation}
\end{lem}
\begin{proof}
Let $f:E(S,v)\rightarrow Ang(T)$ and $g:E(S\diamond_{v}^{f}T,w)\rightarrow
Ang(U)$. We look for $\wt g:E(S,w)\rightarrow Ang(U)$ and
$\wt f:E(S\diamond_{w}^{\wt g}U,v)\rightarrow Ang(T)$ such that the
Equation  \eqref{equa} is verified. Let  $\wt g$ be the restriction of $g$
on the edges $e$ from $S$ and $\wt f=f$. Here $E(S\diamond_{w}^{g}U,v)=E(S,v)$
because the vertices $v$ and $w$ are disjoint.
\end{proof}
\textit{Proof of Theorem \ref{principal} (end):} Thus to show  disjoint
associativity, it remains to show that for any pair $(f,g)$ and $(\wt f,\wt g)$ we
have $C(f,g)=D(\wt f, \wt g)$. We set
$C'(f,g)=C(f,g)+d\big( (S\diamond_{v}^{f_0}T)\diamond_{w}^{g_0}U\big)$ and $D'(\wt
f,\wt g)=D(\wt f,\wt g)+d\big((S\diamond_{w}^{\wt g_0}U)\diamond_{v}^{\wt
f_0}T\big)$.
 We have:
\begin{eqnarray*}
k(\wt g)&=&d(S\diamond_{w}^{\wt g}U)-d(S\diamond_{w}^{\wt g_0}U)\\
&=&\sum_{e\in E(S,w)}h\big(\wt g(e)\big).|B_e|,
\end{eqnarray*}
where $h\big(\wt g(e)\big)$ is the distance between the root of $U$ and $\wt g(e)$
in the new tree $S\diamond_w^{\wt g}U$.
 We also have:
\begin{equation*}
d\big((S\diamond_w^{\wt g}U)\diamond_v^{\wt f_0}T\big)-d\big((S\diamond_w^{\wt
g_0}U)\diamond_v^{\wt f_0}T\big)=k(\wt g)
\end{equation*}
because we changed the vertex $v$ by a tree of the same weight, and because $\wt
f_0$ is defined by grafting onto the root. This proves:
\begin{equation*}
D'(\wt f, \wt g)=d\big((s\diamond_w^{\wt g}U)\diamond_v^{\wt f}T\big).
\end{equation*}
By Lemma \ref{lemme disjointe} again, $C'(f,g)=D'(\wt f,\wt g)$ and  $C(f,g)=D(\wt
f,\wt g)$,
 which proves disjoint associativity.
\end{itemize}
 The partial compositions defined on  $\mathcal{O}^{\lambda}$ hence verify the
axioms of a current-preserving operad.
\end{proof}

\begin{rmk}
With the notations of Sec \ref{ordinaryoperad}, the current-preserving
operad $\mathcal{B}^\lambda$ is naturally associated with $({\mathcal{B}^0})^{\mathbb{N}^*}$
for $\lambda=0$ and with $Brace^{\mathbb{N}^*}$ for $\lambda=1$.
\end{rmk}

\begin{thm} The map $\varphi$ introduced in Definition \ref{symm} is a
morphism of current-preserving operad from $\mathcal{O}^\lambda$ to $\mathcal{B}^\lambda$ i.e:
$$\varphi(S\circ_{v,\lambda}T)=\varphi(S)\diamond_{v,\lambda}\varphi(T).$$
\end{thm}
\begin{proof}
By definition of partial compositions on planar trees and non-planar trees,
and as the shuffle product permute the branches in all possible ways, then we verify that $\varphi$ is a current-preserving operad morphism. The key point is the following: for any $f:E(S,v)\to v(T)$, choosing a planar representative of $S\circ_{v}^{f}T$ amounts to choosing planar representatives $\tilde{S}$ and $\tilde{T}$ of $S,T$ respectively, together with an increasing map $\tilde{f}: E(\tilde{S},v)\to Ang(\tilde{T})$ above $f$, i.e such that the following diagram commutes : 
 \diagramme{
\xymatrix{E(\tilde{S},v)\ar[r]^{\tilde{f}}\ar[dr]_f & Ang(\tilde{T})\ar@{->>}[d]\\
&v(T)
}
}

\end{proof}

{\bf{Acknowledgments:}} I would like to thank Dominique Manchon for valuable discussions and comments.

\end{document}